\documentclass{article}

\usepackage{amsmath, amsthm, amsfonts}

\usepackage{amsthm, verbatim, hyperref}

\title{Remnant inequalities and doubly-twisted conjugacy in free groups}
\author{P. Christopher Staecker\thanks{
Address: Department of Mathematics and Computer Science, Fairfield
University, Fairfield CT}
\thanks{Email: cstaecker@fairfield.edu}
}

\theoremstyle{hplain}

\newtheorem{thm}{Theorem}
\newtheorem{lem}[thm]{Lemma}
\newtheorem{cor}[thm]{Corollary}

\theoremstyle{definition}
\newtheorem{exl}[thm]{Example}
\newtheorem{defn}[thm]{Definition}

\newcommand{\BSL}{\emph{BSL}}
\newcommand{\SB}{\emph{SB}}
\newcommand{\vt}{\mathbf{t}}

\DeclareMathOperator{\Rem}{Rem}

\DeclareMathOperator{\Epi}{Epi}
\DeclareMathOperator{\Mono}{Mono}
\DeclareMathOperator{\ED}{ED}
\newcommand{\Z}{\mathbb Z}

\renewcommand{\bar}{\overline}
\renewcommand{\phi}{\varphi}

\begin{document}
\bibliographystyle{plain}

\maketitle

\begin{abstract}
We give two results for computing doubly-twisted conjugacy relations
in free groups with respect to homomorphisms $\phi$ and $\psi$ such
that certain remnant words from $\phi$ are longer than the images of
generators under $\psi$. 

Our first result is a remnant inequality condition which implies that
two words $u$ and $v$ are not doubly-twisted conjugate. Further we
show that if $\psi$ is given and $\phi$, $u$, and $v$ are chosen at
random, then the probability that $u$ and $v$ are not doubly-twisted
conjugate is 1. In the particular case of singly-twisted conjugacy, 
this means that if $\phi$, $u$, and $v$ are chosen at random, then $u$
and $v$ are not in the same singly-twisted conjugacy class with probability 1.

Our second result generalizes Kim's ``bounded solution length''. We
give an algorithm for deciding doubly-twisted conjugacy relations in the
case where $\phi$ and $\psi$ satisfy a similar remnant inequality. In
the particular case of singly-twisted conjugacy, our algorithm
suffices to decide any twisted conjugacy relation if $\phi$ has
remnant words of length at least 2.

As a consequence of our generic properties we give an elementary proof
of a recent result of Martino, Turner, and Ventura, that
computes the densities of injective and surjective homomorphisms
from one free group to another. We further compute the expected value of the density of the image of a homomorphism.
\end{abstract}

\section{Introduction}
Let $G$ and $H$ be finitely generated free groups, and $\phi, \psi: G
\to H$ be homomorphisms. The group $H$ is partitioned into the set of
\emph{doubly-twisted conjugacy} classes as follows: $u, v \in H$ are in the
same class (we write $[u] = [v]$) if and only if there is some $z\in G$
with
\[ u = \phi(z) v \psi(z)^{-1}. \]

Our principal motivation for studying doubly-twisted conjugacy is Nielsen
coincidence theory (see \cite{gonc05} for a survey), the study of
the coincidence set of a pair of mappings and the minimization of this
set while the mappings are changed by homotopies. Our focus on free
groups is motivated specifically by the problem of computing \emph{Nielsen
classes} of coincidence points for pairs of mappings $f, g: X \to Y$,
where $X$ and $Y$ are compact surfaces with boundary. 

A necessary condition for two coincidence points to be combined by a
homotopy (thus reducing the total number of coincidence points) is
that they belong to the same Nielsen class. (Much
of this theory is a direct generalization of similar techniques in
fixed point theory, see \cite{jian83}.) The number of ``essential''
Nielsen classes is called the Nielsen number, and is a lower bound for
the minimal number of coincidence points when $f$ and $g$ are allowed
to vary by homotopies. 

On surfaces with boundary, deciding when two coincidence points are in the
same Nielsen class is equivalent to solving a natural doubly-twisted
conjugacy problem in the fundamental groups, using the induced
homomorphisms given by the pair of mappings. 
Thus the Nielsen classes of coincidence points correspond to
twisted conjugacy classes in $\pi_1(Y)$. 

The problem of computing doubly-twisted conjugacy classes in free groups is
nontrivial, even in the singly-twisted case which arises in fixed
point theory, where $\phi$ is an endomorphism and $\psi$ is the
identity. An algorithm for the singly-twisted conjugacy decision
problem where $\phi$ is an automorphism is given in \cite{bmmv06}.
Few techniques for computing doubly-twisted conjugacy in free groups are
available. A generally applicable technique using abelian and
nilpotent quotients is given in \cite{stae07b}, but is it is hard to
predict when it will be sucessful. A technique is given in
\cite{stae09b} which can often show 
that two words are in different doubly-twisted conjugacy classes, but
the hypotheses on the homomorphisms are quite strong.

Our methods are based on the combinatorial \emph{remnant} condition
for homomorphisms. Informally, a homomorphism $\phi$ has remnant if the
images under $\phi$ of generators have limited cancellation when
multiplied together. The noncancelling parts are called the remnant
subwords.

If $\phi, \psi: G \to H$ are
homomorphisms and $u,v \in H$ are words, we consider the remnant subwords
of $\phi$ which remain uncancelled even after making products with the
words $u$ and $v$ themselves. If, in each generator $a$, this remnant
subword is of length greater than or equal to the length of
$\psi(a)$, then $u$ and $v$ are in different doubly-twisted conjugacy
classes. This is shown in Section \ref{densitysection}. 

In Section \ref{mtvsection} we present some improvements
to recent work of Martino, Turner, and Ventura in
\cite{mtv08} concerning the density of injective and surjective
homomorphisms of free groups. Their paper shows that, when the rank of
$H$ is greater than 1, a homomorphism chosen at random will be
injective but not surjective with probability 1. We give a new and
more elementary proof of this theorem, and strengthen the result
concerning surjectivity by showing that the expected value of the
density of the image subgroup of a random
homomorphism is 0. We also treat the case when the rank of $H$ is 1.

In Section \ref{bslsection} we consider only the traditional remnant
subword (without making products with $u$ and $v$). If this subword of
$\phi(a)$ is of 
length strictly greater than the length of $\psi(a)$ for each
generator $a$, then we give an algorithm to decide whether or not $u$
and $v$ are in different doubly-twisted conjugacy classes. This result
is a generalization of Kim's ``bounded solution
length'' technique in \cite{kim07}, and
was developed independently for singly-twisted conjugacy by Hart,
Heath, and Keppelmann in \cite{hhk09}.

In Sections \ref{densitysection} and \ref{bslsection} we show that,
given a homomorphism $\psi$, the remnant inequality used in that
section will hold with probability 1 when $\phi$ 
is chosen at random. This implies in Section \ref{densitysection} that
if $\psi$ is fixed and $\phi$, $u$, and $v$ are all chosen at random,
then $[u] \neq [v]$ with probability 1. In Section \ref{bslsection} we show that
if $\psi$ is fixed and $\phi$ is chosen at random, then there is an
algorithm to decide whether or not $[u] = [v]$ for any words $u$ and $v$. 

The techniques and algorithms described in this paper have been
implemented for the computational algebra system GAP
\cite{gap}. Source code and a user-friendly web based version are
available for experimentation at the author's website.  

The author would like to thank Robert F. Brown, Marlin Eby, and Philip
Heath for helpful comments on this paper, and Armando Martino for
bringing the reference \cite{mtv08} to our attention.

\section{Generic remnant properties}\label{reviewsection}
Wagner, in \cite{wagn99}, defined the remnant condition for
free group endomorphisms which would become a key tool for several
later techniques for computation of the Nielsen number in fixed point
theory (the special case where $\psi$ is the identity) for certain
mappings on surfaces with boundary. Extensions of Wagner's technique
have been made in \cite{hk07} and \cite{kim07}, and for Nielsen
periodic point theory in \cite{hhk08}. 

Wagner's definition of remnant extends to homomorphisms (not necessarily
endomorphisms) of free groups as follows. Throughout, for a word $w
\in H$, the reduced word length of $w$ is denoted $|w|$. 

\begin{defn}\label{remnantdef}
Let $H$ be a free group, and $t = (h_1, \dots, h_n)$ a tuple of words
of $h$. We say that $t$ \emph{has remnant} when, for each $i$, there
is a nontrivial subword of $h_i$ which does not cancel in any product
of the form
\[ h_j^{\alpha_j} h_i h_k^{\beta_k} \]
where $j,k \in \{1, \dots, n\}$, with $\alpha_l, \beta_l \in \{ -1, 0,
1 \}$ for $l\neq i$ and $\alpha_i, \beta_i \in \{ 0, 1\}$. Each such
noncanceling subword is called the remnant of $h_i$, denoted $\Rem_t
h_i$.

The statement that a set of elements has remnant is closely related to
the statement that it be Nielsen reduced (see e.g.\ \cite{ls77}).  

If $G$ has generators $a_1, \dots, a_n$, and $\phi: G \to H$ is
a homomorphism, then we say that $\phi$ \emph{has remnant} if the
tuple $(\phi(a_1), \dots, \phi(a_n))$ has remnant in the above
sense. In this case, the remnant of $\phi(a_i)$ is denoted $\Rem_\phi
a_i$. 

If $\phi$ has remnant, the \emph{remnant length} of $\phi$ is the
minimum length of $|\Rem_\phi a_i|$ for any $i$.
\end{defn}

Throughout this paper, the remnant condition is used 
typically as follows: we will have a homomorphism $\phi: G \to H$ and some
element $z\in G$ with reduced form $z = a_{j_1}^{\eta_1} \dots
a_{j_k}^{\eta_k}$. When $\phi$ has remnant, writing
$X_i = \phi(a_{j_i})^{\eta_i}$, we have the reduced product
\[ \phi(z) = X_1 \dots X_k = R_1 \dots R_k, \]
where $R_i$ is the subword of $X_i$ which does not cancel in the
product $X_1 \dots X_k$. Since $\phi$ has remnant, the right hand side
will be fully reduced and $(\Rem_\phi a_i)^{\eta_i}$ will be a subword
of $R_i$ for all $i$. 

Use of the asymptotic density in the context of doubly-twisted
conjugacy was presented in \cite{stae09b}. We will quote the relevant
terms and results here.

For a free group $G$ and a natural number $p$, let $G_p$ be the subset
of all words of length at most $p$. The \emph{asymptotic density} (or simply
\emph{density}) of a subset $S\subset G$ is defined as 
\[ D(S) = \lim_{p \to \infty} \frac{|S \cap G_p|}{|G_p|}, \]
where $|\cdot|$ denotes the cardinality. The set $S$ is said to be
\emph{generic} if $D(S) = 1$.

Similarly, if $S \subset G^l$ is a set of $l$-tuples of
elements of $G$, the asymptotic density of $S$ is defined as
\[ D(S) = \lim_{p \to \infty} \frac{|S \cap (G_p)^l|}{|(G_p)^l|}, \]
and $S$ is called \emph{generic} if $D(S) = 1$.

A homomorphism on the free groups $G \to H$ with $G = \langle a_1,
\dots, a_n\rangle$ 
is equivalent combinatorially to an $n$-tuple of elements of $G$
(the $n$ elements are the words $\phi(a_1), \dots, \phi(a_n)$). Thus
the asymptotic density of a set of 
homorphisms can be defined in the same sense as above, viewing the set of
homomorphisms as a collection of $n$-tuples. 

A theorem of Robert F.\ Brown in \cite{wagn99} established that
``most'' endomorphisms have remnant. This is strengthened and made
more specific in \cite{stae09b} as the following:
\begin{lem}\label{genericlem} 
Let $G$ and $H$ be free groups with the rank of $H$ greater than 1. Then
for any natural number $l$, the set of homomorphisms
  $\phi:G \to H$ with remnant length at least $l$ is generic.
\end{lem}

\section{The density of non-conjugate pairs}\label{densitysection}
Given a homomorphism $\phi: G \to H$ and two elements $u,v \in H$,
there is a natural homomorphism $\phi * u * v: G * \Z * \Z \to H$,
(where $*$ denotes the free product) defined as follows: let $G =
\langle a_1, \dots, a_n \rangle$, 
write $b_1$ as the generator of the first $\Z$ factor, and $b_2$ as
the generator of the second $\Z$ factor. We then define $\phi*u*v$ on
each factor by
\[ \phi*g*h : \begin{array}{rcl}
a_i &\mapsto &\phi(a_i) \\
b_1 &\mapsto &u \\
b_2 &\mapsto &v \end{array} \]

The basic idea of the present theorem is inspired by the nice proof of
Wagner's theorem given in \cite{kim07}. 

\begin{thm}\label{classdist}
Let the rank of $H$ be greater than 1, let $\phi, \psi: G \to H$ be
homomorphisms, and let $u, v \in H$. Let $\bar \phi = \phi * u *
v$. If $\bar \phi$ has remnant with 
\[ |\Rem_{\bar \phi} a| \ge |\psi(a)| \]
for all generators $a\in G$, then $[u]\neq [v]$.
\end{thm}

\begin{proof}
For the sake of a contradiction we assume that $[u]=[v]$, and so there
is some $z \in G$ with 
\begin{equation}\label{twistconj}
\psi(z) = u^{-1} \phi(z) v. 
\end{equation}
Express $z$ as the reduced word
\[ z = a_{j_i}^{\eta_i} \dots a_{j_k}^{\eta_k}, \]
where $a_{j_i}$ are generators of $G$ and $\eta_i \in \{-1, 1\}$, and
write the images of the generators as the reduced words $X_i = \phi(a_{j_i}^{\eta_i})$ and $Y_i = \psi(a_{j_i}^{\eta_i})$.

We first write
\[ u^{-1}\phi(z)v = u^{-1}X_1 \dots X_kv = R_u R_1 \dots R_k R_v, \]
where the right hand side is reduced, $R_i$ is the subword of
$X_i$ which does not cancel in the above, and $R_u$ and $R_v$ are the
subwords of $u^{-1}$ and $v$ which do not cancel in the above. Because
$\bar \phi$ has remnant, we know that $R_u$ and $R_v$ are nontrivial,
and $R_i$ contains $(\Rem_{\bar\phi} a_{j_i})^{\eta_i}$ as a subword.

Then we have  
\begin{align*}
|u^{-1}\phi(z) v| &= |R_u| + |R_v| + \sum_{i=1}^{k} |R_i| \\
&\ge |R_u| + |R_v| + \sum_{i=1}^{k} |\Rem_{\bar\phi}
a_{j_i}| \\
&> \sum_{i=1}^{k} |\Rem_{\bar\phi} a_{j_i}| \ge \sum_{i=1}^k |Y_i|
\ge |\psi(z)|,
\end{align*}
and the strict inequality contradicts \eqref{twistconj}.
\end{proof}

The above theorem is similar to the result of \cite{stae09b} that
$[u]\neq [v]$ when $\phi * \psi * (uv^{-1})$ has remnant. 
The fact that we require no
remnant condition of $\psi$ allows our result, unlike that of
\cite{stae09b}, to be specialized to the case of singly-twisted
conjugacy: 
\begin{cor}
For an endomorphism $\phi: G\to G$, let $[u]$ denote the singly-twisted conjugacy class of $u$. 

If $\phi * u * v$ has remnant, then $[u] \neq [v]$.
\end{cor}

We now note that Theorem \ref{classdist} will generically apply for a
particular $\psi$ when $\phi$, $u$, and $v$ are chosen at random. This
result is a strengthening of the final Theorem of \cite{stae09b}
using different methods. 

\begin{thm}\label{generictriple}
Let the rank of $H$ be greater than 1, and let $\psi: G \to H$ be any
homomorphism. Then (again letting $[u]$ denote the doubly-twisted
conjugacy class with respect to $\phi$ and $\psi$) the set 
\[ S = \{ (\phi, u, v) \mid [u] \neq [v] \} \]
is generic.
\end{thm}
\begin{proof}
Letting $n$ be the number of generators of $G$, we may view a triple
$(\phi, u, v)$ as a $(n+2)$-tuple of elments of $H$ (since a choice of
$\phi$ is combinatorially equivalent to a choice of the $n$ image
words in $H$). Thus we view $S$ as a subset of the cartesian product
$H^{n+2}$.

Let $k$ be the maximum length of any $|\psi(a_i)|$. Then by Theorem \ref{classdist} we have
\begin{align*}
S &= \{ t = (\phi(a_1), \dots, \phi(a_n), u, v) \mid [u] \neq [v] \}
\supset \{ t \in H^{n+2} \mid t \text{ has remnant length at least $k$} \} \\
&= \{ \rho: F_{n+2} \to H \mid \rho \text{ has remnant length at least
  $k$} \} 
\end{align*}
and by Lemma \ref{genericlem} this set is generic.
\end{proof}

Since $\psi$ above is allowed to be any homomorphism, the special
cases where $\psi$ is the identity homomorphism and the trivial
homomorphism give:

\begin{cor}\label{triplecor}
Let $G$ and $H$ be free groups with the rank of $H$ greater than 1.
\begin{itemize}
\item If $G=H$ and $[u]$ denotes the singly-twisted conjugacy class of $u
\in G$ with respect to an endomorphism $\phi:G \to G$, then the set
\[ \{ (\phi, u, v) \mid [u] \neq [v] \} \] 
is generic.
\item For homomorphisms $\phi: G \to H$, the set
\[ \{ (\phi, u, v) \mid uv^{-1} \not\in \phi(G) \} \]
is generic.
\end{itemize}
\end{cor}
\begin{proof}
The first statement follows directly from Theorem \ref{generictriple}
letting $\psi$ be the identity. For the second statement, let $\psi$
be the trivial homomorphism. Then $[u]=[v]$ if and only if $u =
\phi(z)v$, which is to say that $uv^{-1} \in \phi(G)$.
\end{proof}

\section{The densities of injections and surjections}\label{mtvsection}
From the second statement of Corollary \ref{triplecor}, together with
Lemma \ref{genericlem}, we obtain an alternative (and easier) proof of
a recent result by Martino, Turner, and Ventura in \cite{mtv08}
concerning the densities of injective and surjective homomorphisms of
free groups. The preprint \cite{mtv08} addresses only the second
statement below:
\begin{thm}\label{mtvthm}
Given free groups $G$ and $H$, let
$\Epi(G,H)$ and $\Mono(G,H)$ be the sets of all surjective and
injective homomorphisms $G \to H$, respectively. 

\begin{enumerate}
\item \label{one1} If the rank of $H$ is 1 and the rank of $G$ is
$n>1$, then
\[ D(\Epi(G,H)) = \frac{1}{\zeta(n)}, \quad D(\Mono(G,H))
= 0, \]
where $\zeta(n)$ is the Reimann zeta function. (Note that $\zeta(n)
\to 1$ as $n \to \infty$.)
\item \label{other} If the ranks of $G$ and $H$ are both 1, or the rank
  of $H$ is greater than 1, then 
\[ D(\Epi(G,H)) = 0, \quad D(\Mono(G,H)) = 1. \]
\end{enumerate}
\end{thm}
\begin{proof}
Let $G$ have generators $a_1, \dots, a_n$.

Statement \ref{one1} concerns homomorphisms $G \to \Z$, each of which
is equivalent to a choice of $n$ integers (where $n>1$). Let a
homomorphism $\phi$ be given by integers $m_1, \dots, m_n$. 
Then $\phi$ is never injective: if $\phi(a_1) = m_1$ and
$\phi(a_2) = m_2$, then $\phi(a_1^{m_2}) = m_1m_2 =
\phi(a_2^{m_1})$. Thus $\Mono(G,H)$ is empty, and so $D(\Mono(G,H)) =
0$.

The homomorphism
$\phi$ is surjective if and only if there are integers $k_1, \dots,
k_n$ with 
\[ k_1m_1 + \dots + k_nm_n = 1. \]
This in turn is equivalent to requiring that $\gcd(m_1, \dots, m_n)=1$,
since the gcd is the smallest positive integer which is an integral
linear combination of $m_1, \dots, 
m_n$. It is known that the probability (in the appropriate asymptotic
sense) of $n$ randomly chosen integers being coprime is $1/\zeta(n)$,
see \cite{nyma72}. Thus $D(\Epi(G,H)) = 1/\zeta(n)$. 

Now we prove statement \ref{other}, first in the case where the ranks
of $G$ and $H$ are both 1. 
We may consider both $G$ and $H$ to be the
integers $\Z$. In this case, a homomorphism is equivalent to a choice
of a single integer (the degree of the homomorphism). The homomorphism
will be surjective if and only if the degree is $\pm 1$, and will be
injective if and only if the degree is nonzero. The desired densities
follow.

Now we prove the case where the rank of $H$ is greater than 1. We
first note that any homomorphism $\phi: G \to H$ with remnant is
injective: If some $\phi$ with remnant were not injective, then there
would be words $x, y \in G$ with  
\[ \phi(x)\phi(y)^{-1} = 1. \]
But writing the above in terms of generators will show that the above
product cannot cancel, since the remnants will remain. Since the
homomorphisms with remnant have density 1 by Lemma \ref{genericlem}, we
have $D(\Mono(G,H)) = 1$.

The statement concerning $\Epi(G,H)$ is implied by the second
statement of Corollary \ref{triplecor}.
Let $S$ be the set of triples $(\phi, u, v)$ with $uv^{-1} \not \in
\phi(G)$, and Corollary \ref{triplecor} says that $D(S) = 1$. Let $T$ be
the set of non-surjective homomorphisms $G \to H$. Certainly $S$ is a
subset of $T \times H \times H$, and it is easy to check that $D(A
\times H) = D(A)$ for any set $A$. Thus we have
\[
D(\Epi(G,H)) = 1 - D(T) = 1 - D(T \times H \times H) \le 1-D(S) = 0.
\]
\end{proof}

The second statement of Corollary \ref{triplecor} suggests that a much
stronger statement concerning surjective homomorphisms may be
possible. In the case where $H$ has rank greater than 1, we have shown
that the image set $\phi(G)$ is a proper subset of $H$ with probability
1. We wish to give a more specific measure of the generic size of the
subset $\phi(G) \subset H$. 

Let $\ED(G,H)$ be the \emph{expected value} of $D(\phi(G))$, which we define
as follows:
\[ \ED(G,H) = \lim_{p \to \infty} \frac{1}{|H_p|^n} \sum_{\phi \in
  (H_p)^n} D(\phi(G)), \]
where $n$ is the number of generators of $G$, so we regard a
homomorphism $\phi:G \to H$ as an element of $H^n$. 

This expected value is related to $D(\Epi(G,H))$ as follows: a
surjection $\phi$ has $D(\phi(G)) = 1$, and so, letting $E_p =
\Epi(G,H) \cap (H_p)^n$, we have
\[
\ED(G,H) \ge \lim_{p \to \infty} \frac{1}{|H_p|^n} \sum_{\phi \in
  E_p} 1 = \lim_{p \to \infty} \frac{|E_p|}{|H_p|^n} =
D(\Epi(G,H)). 
\]

Thus Theorem \ref{mtvthm} shows that if $H$ has rank 1 and the rank of
$G$ is $n>1$, then $\ED(G,H) \ge 1/\zeta(n)$. Theorem \ref{mtvthm}
gives no information in the case where $H$ has rank greater than 1,
since it would only imply that $\ED(G,H) \ge 0$, which is already clear.
Our goal for the remainder of the section is to compute the precise
value of $\ED(G,H)$. We rely on a lemma which estimates $D(\phi(G))$ when
$\phi$ has remnant.

\begin{lem}\label{poslem}
Let $\phi:G \to H$ be a homomorphism with remnant length $l$, and let
$n>1$ be the number of generators of $H$. Then  
\[ D(\phi(G)) \le 16n(2n-1)^{\lceil l/2 \rceil}. \]
\end{lem}
\begin{proof}
If $\phi$ does not have remnant, then $l = 0$ and the statement to be
proved is $D(\phi(G)) \le 16n$, which is always the case. Thus we
assume that $\phi$ has remnant, and so $l>0$.

If $w \in \phi(G)$, then there is some $z$ with $w\phi(z) = 1$.
Thus if $w \in \phi(G)$, then $\phi * w$ does not have
remnant. Letting 
\[ S = \{ w \mid \phi * w \text{ does not have remnant}\},  \]
we have $\phi(G) \subset S$. We will give an upper bound on $D(S)$,
which will imply an upper bound on $D(\phi(G))$.

Let $G$ have generators $a_1, \dots, a_n$, and let $u_i$ and $v_i$ be
the subwords of $\phi(a_i)$ respectively ``before'' and ``after'' the
remnant subword. That is, we can write $\phi(a_i)$ as 
\[ \phi(a_i) = u_i r_i v_i, \]
where $r_i = \Rem_\phi a_i$, and the product $u_i r_i v_i$ is reduced
(allowing perhaps $u_i$ and $v_i$ to be trivial). Write $r_i =
s_it_i$, where $|s_i|$ and $|t_i|$ are at most $\lceil |r_i|/2 \rceil
\ge \lceil l/2 \rceil$. For brevity below, write $k = \lceil l/2 \rceil$.

In order for $\phi * w$ to have no remnant, some initial 
subword of $w$ or $w^{-1}$ must equal one of the words $u_is_i$ or
$(t_iv_i)^{-1}$, or some terminal subword of $w$ or $w^{-1}$ must equal
$(u_is_i)^{-1}$ or $t_iv_i$. (Note that all of these words have length
greater than $k$.)

If $x$ is a word of length $m<p$, the number of words $w$ of length $p$
having $x$ as the initial subword is $(2n-1)^{p-m}$, since the first
$m$ letters of $w$ are fixed, and the remaining $p-m$ letters can be
any letter of $H$ except the inverse of the previous. 
Thus the number of 
words $w$ having $u_is_i$ as the initial subword is
$(2n-1)^{p-|u_is_i|} \le (2n-1)^{p-k}$. Similarly the number of words $w$ of
length $p$ having $(t_iv_i)^{-1}$ as the initial subword is at most
$(2n-1)^{p-k}$, and the number of words $w$ having $(u_is_i)^{-1}$ as the
terminal subword and the number of words $w$ having $t_iv_i$ as the
terminal subword are at most $(2n-1)^{p-k}$. Since there are $n$ possible
values for $i$, and we must allow for words $w^{-1}$ with each of the
above 4 constraints, we have
\[ |S \cap H_p| \le 8n(2n-1)^{p-k}, \]
and thus
\begin{equation}\label{Sdensity} 
D(S) \le \lim_{p \to \infty} 8n\frac{(2n-1)^{p-k}}{|H_p|} 
\end{equation}

For $i>0$, the number of words of length exactly $i$ is $2n(2n-1)^{i-1}$, 
since the first letter can be any letter of $H$, while each subsequent
letter can be anything but the inverse of the previous. Summing gives
the formula 
\[
|H_p| = 1 + \sum_{i=1}^p 2n(2n-1)^{i-1} = \frac{n(2n-1)^p - 1}{n-1}
\]
and thus we have
\[
\frac{(2n-1)^{p-k}}{|H_p|} = \frac{(n-1)(2n-1)^{p-k}}{n(2n-1)^p - 1}
\le 2 \frac{(n-1)(2n-1)^{p-k}}{n(2n-1)^p} \le 2(2n-1)^{-k}.
\]

Then \eqref{Sdensity} becomes
\[ D(S) \le 16n(2n-1)^{-k}, \]
and the fact that $\phi(G) \subset S$ gives the desired result.
\end{proof}

We now make the computation of the expected value. 
\begin{thm}\label{edthm}
Let $G$ and $H$ be free groups. 
\begin{enumerate}
\item If the rank of $H$ is 1 and the rank of $G$ is $n>1$, then
\[ \ED(G,H) = \frac{\zeta(n+1)}{\zeta(n)} \]
\item If the ranks of $G$ and $H$ are both 1, or if the rank of $H$ is
  greater than 1, then 
\[ \ED(G,H) = 0. \]
\end{enumerate}
\end{thm}
\begin{proof}
We begin with the first statement, which concerns homomorphisms $G \to
\Z$, each of which is equivalent to a choice of $n$ integers (with
$n>1$). If $\phi$ is given by the tuple of integers $\vt = (m_1, \dots,
m_n)$, then $\phi(G)$ is equal to the set $d\Z$, the integer multiples
of $d$, where $d = \gcd(\vt)$. This set has density
$1/\gcd(\vt)$. Thus we have:
\[ \ED(G,\Z) = \lim_{p \to \infty} \frac{1}{|H_p|^n} \sum_{
  \vt \in (H_p)^n} \frac{1}{\gcd(\vt)} \]

Thus $\ED(G,\Z)$ is equal to the expected value of the reciprocal of
the $\gcd$ function when applied to $n$ arguments. A standard
rearrangement expresses this expected value as a series:
\begin{align*}
\ED(G,\Z) &= \lim_{p \to \infty} \frac{1}{|H_p|^n} \sum_{\vt
\in (H_p)^n} \frac{1}{\gcd(\vt)} \\
&= \lim_{p \to \infty} \frac{1}{|H_p|^n} \sum_{d = 1}^p \frac{1}{d}
|\{ \vt \in (H_p)^n \mid \gcd(\vt) = d \}| \\
&= \sum_{d=1}^\infty \frac1d \lim_{p \to \infty}
\frac{1}{|H_p|^n} |\{ \vt \in (H_p)^n \mid
\gcd(\vt) = d\}|, 
\end{align*}
where the exchanging of the limit and the sum is valid provided that
the inner limit exists and is finite. The inner limit can be interpreted as 
the probability (in the appropriate asymptotic
sense) that a random $n$-tuple has gcd equal to $d$. 
This probability
is known to be equal to $d^{-n}/\zeta(n)$ (see equation 5.1 of
\cite{cs87}). This gives 
\[ \ED(G,\Z) = \sum_{d=1}^\infty \frac{d^{-1-n}}{\zeta(n)} =
\frac{1}{\zeta(n)} \sum_{d=1}^{\infty} \frac{1}{d^{n+1}} =
\frac{\zeta(n+1)}{\zeta(n)} \]
where the last equality is the definition of $\zeta(n+1)$. 

Now we prove the second statement. First we treat the case where $G$
and $H$ are rank 1. Then we will write $G =H=\Z$, and a homomorphism
$\phi:\Z \to \Z$ is equivalent to a single integer (the degree of
$\phi$). Writing $\phi \in \Z$ as this integer, we have
\[ \ED(\Z,\Z) = \lim_{p \to \infty} \frac{1}{2p+1} \sum_{\phi = -p}^p
D(\phi(\Z)). \]

If $\phi$ is the homomorphism given by multiplication by $k$, then the
image $\phi(\Z)$ is the set $k\Z = \{kn \mid n \in \Z \}$, which has
density $1/|k|$. Thus the above becomes
\[ \ED(\Z,\Z) = \lim_{p \to \infty} \frac{1}{2p+1} 2\sum_{k=1}^p
\frac{1}{k}. \]
(We have dropped the $k=0$ term, since the image of the trivial
homomorphism has density 0.) It is routine to verify that the above
limit exists and equals 0.

Now we prove the case where the rank of $H$ is greater than 1.
Let $R_l$ be the set of all homomorphisms $\phi:G \to H$ with remnant
length at least $l$. By Lemma \ref{genericlem} this set is generic.

Let $n$ be the rank of $G$, and let $l$ be any natural number. Then,
using Lemma \ref{poslem}, we have
\begin{align*}
\ED(G,H) &= \lim_{p \to \infty} \frac{1}{|H_p|^n} \left( \sum_{\phi \in
  R_l \cap (H_p)^n} D(\phi(G)) + \sum_{\phi \in (H_p)^n - R_l}
D(\phi(G)) \right) \\
&\le \lim_{p \to \infty} \frac{1}{|H_p|^n} \left( \sum_{\phi \in R_l
  \cap (H_p)^n} 16n(2n-1)^{-\lceil l/2 \rceil} + \sum_{\phi \in (H_p)^n
  - R_l} 1 \right) \\
&= \lim_{p \to \infty} 16n(2n-1)^{-\lceil l/2 \rceil} \frac{|R_l \cap
  (H_p)^n|}{|H_p|^n} + \frac{|(H_p)^n - R_l|}{|H_p|^n} 
\end{align*}
In the limit as $p \to \infty$, the first fraction converges to the
density of $R_l$, which is 1, and the second converges to the density
of its complement, which is 0. Thus we have
\[ \ED(G,H) \le 16n(2n-1)^{-\lceil l/2 \rceil}, \]
and since $l$ is any natural number, we have $\ED(G,H) = 0$.
\end{proof}

\section{Bounded solution length}\label{bslsection}
In this section we show that a remnant inequality similar to the one
in Theorem \ref{classdist} implies an algorithm for deciding
doubly-twisted conjugacy relations.

\begin{defn}
Given homomorphisms
$\phi, \psi: G \to H$ and a pair of group elements $u,v \in H$, we say
that the pair $(u,v)$ has \emph{bounded solution length} (or \BSL) if
there is some $k>0$ such that the equation $u =\phi(z)v\psi(z)^{-1}$
is satisfied (if at all) only if $|z|\le k$. The smallest such $k$ is
called the \emph{solution bound} (or \SB) for $(u, v)$.
\end{defn}

Our casting of the \BSL\ condition generalizes
the concept for singly-twisted conjugacy of the same name in
\cite{kim07}, where $G=H$ and $\psi$ is assumed to be the identity
homomorphism. (Kim works in the setting where the words $u$ and $v$ are
always taken to be ``Wagner tails'' of $\phi$ which are not 
indirectly related in Wagner's algorithm. We omit this distinction
so that the \BSL\ condition can be defined without any reference to the
set of Wagner tails.)

Our main theorem in this section is that if $\phi$ and $\psi$ satisfy
a remnant condition similar to the condition in Theorem
\ref{classdist} then any pair $(u,v)$ will have \BSL\ with a
predictable solution bound. This implies an algorithm for deciding
doubly-twisted conjugacy relations between any elements. 

The following theorem was independently proved in the setting of
singly-twisted conjugacy by Hart, Heath, and Keppelmann in
\cite{hhk09} using essentially the same argument. Their solution bound
was better than the one initially discovered by this author, and has
been incorporated into the proof below. 

\begin{thm} \label{bsl}
Let $\phi, \psi: G \to H$ be homomorphisms, such that
\[ |\Rem_\phi a_i| > |\psi(a_i)| \]
for each generator $a_i\in G$. Let $l = \min_i(|\Rem_\phi a_i| -
|\psi(a_i)|)$.  Then any pair $(u,v)$ has \BSL, with solution bound 
\[ \SB \le \frac{|u| + |v|}{l}. \]
\end{thm}
\begin{proof}
Let $u, v \in H$, and let $z \in G$ be a word of length $k$. To show
that $(u,v)$ has \BSL, we will show that for $k$ sufficiently
large, we have $\psi(z) \neq u^{-1} \phi(z) v$.

As in the proof of Theorem \ref{classdist}, write $z$ as the reduced
word $z = a_{j_1}^{\eta_1} \dots a_{j_k}^{\eta_k}$, where each $a_j$
is a generator of $G$ and each $\eta_i = \pm 1$. Let $X_i =
\phi(a_{j_i}^{\eta_i})$ and $Y_i = \psi(a_{j_i}^{\eta_i})$, then we 
have
\[ u^{-1} \phi(z) v = u^{-1} X_1 \dots X_k v. \]
Since $\phi$ has remnant, we can use notation as in the proof of
Theorem \ref{classdist} (though this time not worrying about $u^{-1}$
and $v$) to write this product as
\[
u^{-1} \phi(z) v = u^{-1} R_1 \dots R_k v, 
\]
where each $R_i$ is a subword of $X_i$ with $|R_i| \ge
|\Rem_\phi a_{j_i}|$, and no cancellation occurs in any $R_i R_{i+1}$. 

Now we will show that $\psi(z) \neq u^{-1} \phi(z) v$ by showing that
these two words are of different lengths for sufficiently large $k$. We have
\begin{align*}
|u^{-1}\phi(z)v| - |\psi(z)| &= |u^{-1} R_{1} \dots R_{k} v|
- |Y_1 \dots Y_k| \\
&\ge |R_{1} \dots R_{k}| - |u| - |v| - |Y_1 \dots Y_k| \\
&=  - |u| - |v| + \left(\sum_{i=1}^k |R_i|\right) - |Y_1 \dots Y_k| \\
&\ge -|u| - |v| + \sum_{i=1}^k \left( |R_i| - |Y_i| \right) \\
&\ge -|u| - |v| + \sum_{i=1}^k \left( |\Rem_\phi a_{j_i}| - |Y_i| \right) .
\end{align*}
By the hypothesis to our theorem, we know that $|\Rem_\phi a| - |\psi(a)|
\ge l$ for every generator $a \in G$. Therefore the above inequalies give 
\[ |u^{-1}\phi(z)v| - |\psi(z)| \ge kl - |u| - |v|, \]
and we can choose $k$ sufficiently large so that $|u^{-1} \phi(z) v| -
|\psi(z)|$ is greater than zero. In particular it suffices to choose
\[ k > \frac{|u| + |v|}{l} \]
which is the desired solution bound.
\end{proof}

Theorem \ref{bsl} implies (subject to the remnant hypotheses)
that, given any elements $u, v \in H$, we can algorithmically
determine whether or not $[u] = [v]$:
check for equality of $u = \phi(z) v \psi(z)^{-1}$ where $z$ ranges
over all elements of $G$ with $|z|\le \frac{|u| + |v|}l$. 
\begin{exl}
Let us consider the pair of homomorphisms on the free group with two
generators: 
\[
\phi: \begin{array}{rcl}
a &\mapsto & baba^2 \\
b &\mapsto & a^2b^{-1}ab^3 \end{array}
\quad
\psi: \begin{array}{rcl}
a &\mapsto & b^{-2} \\
b &\mapsto & a \end{array}
\]
We will compare the two classes $[bab]$ and $[b^4a^2]$. (Neither
Theorem \ref{classdist} nor the result of \cite{stae09b} will suffice
to make the comparison.) It is easy to verify that the hypotheses of
Theorem \ref{bsl} are satisfied with $l=3$. Thus, if there is some $z$ with 
\[ \psi(z) = bab \phi(z) (b^4a^2)^{-1}, \]
then it must be the case that $|z| \le \frac{|bab|+|b^4a^2|}{3} = 3$.

It suffices to check all elements $z$ of length at most 3 in the above
equation. Such a check is performed easily by computer, and reveals
that no element $z$ of length at most 3 satisfies the equation. Thus
$[bab] \neq [b^4a^2]$. 
\end{exl}

This gives new and useful
algebraic decision algorithms, even in the cases where $\psi$ is taken
to be the identity or the trivial homomorphism:
\begin{cor}\label{decisioncor}
Let $G$ and $H$ be finitely generated free groups.
\begin{enumerate}
\item (Doubly-twisted version) If $\phi, \psi: G \to H$ are homomorphisms
  with $|\Rem_\phi a| > |\psi(a)|$ for every generator $a \in G$, then
  there is an algorithm to decide whether or not there is some $z$ with
\[ u = \phi(z) v \psi(z)^{-1} \]
for any $u, v \in H$.
\item (Singly-twisted version) If $\phi: G \to G$ is a homomorphism with
  $|\Rem_\phi a| > 1$ for every generator $a \in G$, then there is an
  algorithm to decide whether or not there is some $z$ with 
\[ v = \phi(z) u z^{-1} \]
for any $u,v  \in H$. (This same algorithm is obtained independently by
Hart, Heath, and Keppelmann in \cite{hhk09}.)
\item ($\psi = 1$ version) If $\phi: G \to H$ is a homomorphism with
  remnant, then there is an algorithm to decide whether or not 
\[ w \in \phi(G) \]
for any $w \in H$. 
\end{enumerate}
\end{cor}
The third statement is obtained from the first by letting $\psi$ be the trivial
homomorphism, $v$ the trivial element, and $u=w$.

We note that in Theorem \ref{bsl}, the statement that $(u,v)$
has \BSL\ over any finite set is a direct generalization of Kim's
Theorem 4.7 of \cite{kim07}. If we let $\psi$ be the identity
homomorphism, then our hypothesis is that, for all 
generators $a \in G$, we have $|\Rem_\phi(a)|
> |a| = 1$, which is to say that $|\Rem_\phi(a)| \ge 2$. This is
precisely the hypothesis used in Kim's Theorem 4.7 to show
that any pair of Wagner tails has \BSL. Kim's focus on the set of Wagner tails allows 
his solution bound to be more specific than ours (Kim shows
that $\SB \le 4$ in all cases). 

We conclude by noting that for a particular choice of homomorphism
$\psi$, the remnant hypothesis for Theorem \ref{bsl} is
satisfied for generic $\phi$. Letting $l$ be the maximum length of
$\psi(a)$ for any generator $a \in G$, Lemma \ref{genericlem} 
shows that 
\[ \{ \phi \mid |\Rem_\phi a| > |\psi(a)| \text{ for any generator $a
  \in G$}\} \]
is generic.

\end{document}